\RequirePackage{fix-cm}
\documentclass[smallextended]{svjour3}       
\smartqed  
\usepackage{amsmath}
\usepackage{amssymb}
\usepackage{graphicx}
\usepackage{amsfonts}
\usepackage[noend]{algpseudocode}
\usepackage{booktabs}
\usepackage[sort,numbers]{natbib}
\newtheorem{assumption}{Assumption}
\begin{document}
\title{Comment on ``First-order methods almost always avoid strict saddle points''
\thanks{This research was supported by
	the Beijing Natural Science Foundation under grant Z180005, and
	the National Natural Science Foundation of China under grants 12171021 and 11822103.}
        }


\author{Jinyang Zheng \and Yong Xia
}

\authorrunning{J. Zheng, Y. Xia} 

\institute{
 J. Zheng \and Y. Xia \at
              School of Mathematical Sciences, Beihang University, Beijing,
100191, P. R. China
\email{yxia@buaa.edu.cn (Y. Xia, Corresponding author)}
}

\date{Received: date / Accepted: date}
\maketitle
\begin{abstract}
The analysis on  the global stability of Riemannian gradient descent method in manifold optimization (i.e., it
avoids strict saddle points for almost all initializations)
 due to Lee et al. (Math. Program. 176:311-337) is corrected. Moreover, an explicit condition on the step-size is presented by the newly introduced retraction L-smooth property.
\end{abstract}


\section{Introduction}

Based on results from dynamic system \cite{shub13},
Lee et al. \cite{lee19} first established the global stability of first-order methods for unconstrained optimization, i.e., they can avoid saddle points for almost all initializations. More precisely,
the first-order iteration method $x_{k+1} = g(x_k)$ almost converges to minimizers under the assumptions that $g(x)$ is a local diffeomorphsim and the derivative $Dg(x)$ has a negative eigenvalue at every strict saddle points. Mathematically, we have
\begin{theorem}[\cite{lee19}]\label{thm1}
Given a $d$-dimensional manifold $\mathcal{X}$, and a $C^1$ mapping $g:\mathcal{X} \rightarrow \mathcal{X}$, the set of initial points that converge to an unstable fixed point has measure zero,
	 \begin{center}
	 	$\mu\left(\left\{x_{0}: ~\lim x_{k} \in \mathcal{A}_{g}^{*}\right\}\right)=0$,~ if $\operatorname{det}(\mathrm{D} g(x)) \neq 0$ for all $x \in \mathcal{X}$,
	 \end{center}
	where $\operatorname{det}(\mathrm{D} g(x))$ is the determinant of the matrix representing $\mathrm{D} g(x)$ with respect to an arbitrary basis, then $\operatorname{det}(\mathrm{D} g(x)) \neq 0$ if and only if  $g(x)$ is a local diffeomorphsim, the set of unstable fixed points $\mathcal{A}_{g}^{*}$ is defined as
	\[
	\mathcal{A}_{g}^{*}=\left\{x:~ g(x)=x, ~\max _{i}\left|\lambda_{i}(D g(x))\right|>1\right\},
	\]
and $\lambda_i(\cdot)$ returns the $i$-th largest eigenvalue of matrix $(\cdot)$.
\end{theorem}

As a natural extension, for the problem of minimizing $f(x)$ on the manifold $\mathcal{M}$, Lee et al. \citep[section 5.5]{lee19} established the global stability of
the  Riemannian  gradient method \cite{absil12} with projection-like retraction reads as:
\begin{equation}
x_{k+1}=g(x_k) \triangleq P_{\mathcal{M}}\left(x_{k}-\alpha P_{Tx_{k}M} \left(\nabla f\left(x_{k}\right)\right)\right), \label{RG}
\end{equation}
where $\nabla f$ is the gradient of function $f$, the orthogonal projection operator $P_{\mathcal{N}}$ onto the manifold $\mathcal{N}$ is defined by
$$
P_{\mathcal{N}}(x):=\operatorname{argmin}\{\|x-y\|:~ y \in \mathcal{N}\}.
$$
The strict saddle assumption in manifold version requires that Riemannian Hessian has at least one negative eigenvalue. Following the proof of Theorem \ref{thm1}, Lee et al.  \cite{lee19} verified the unstability of the strict saddle points when calling the iteration \eqref{RG}  by the following two key results.

\begin{proposition}{\bf(\cite[Prop. 8]{lee19})}\label{thm2:1}
	 At a strict saddle point $x^*$, $Dg(x^*)$ has an eigenvalue of magnitude larger than 1.
\end{proposition}

\begin{proposition}{\bf(\cite[Prop. 9]{lee19})}\label{thm2:2}
	 For a compact submanifold $\mathcal{M}$, there is a strictly positive $\alpha$ such that $\operatorname{det}(\mathrm{D} g(x))\neq 0$.
\end{proposition}

We point out in this note that  the proof of Proposition \ref{thm2:2} due to Lee et al. \cite{lee19} is incorrect. In fact, we need additional assumptions in Proposition \ref{thm2:2}. Our goal is to provide such a reasonable assumption together with a correct proof. As a by-product, we can establish an explicit bound of the step-size $\alpha$ for the globally stable iteration \eqref{RG}.

\noindent {\bf Notation.} First, the basic definitions  as well as notation from manifold optimization are listed in Table \ref{t1}. Let $I$ be the identity matrix of order $n$.
For a matrix $A$, $\lambda_i(A)$ and $\sigma_{i}(A)$ denote
the $i$-th largest eigenvalue and singular value, respectively.
The largest and smallest singular values of $A$ are denoted by $\sigma_{\max}(A)$ and  $\sigma_{\min}(A)$, respectively. Then $\sigma_{1}(A)=\sigma_{\max}(A)$, which is equal to the spectral norm $\|A\|$. $\sigma_{\min}(A)=\sqrt{\lambda_{\min}(A^TA)}$. For a vector $a$, $\|a\|$ returns the Euclidean norm.

\begin{table}\label{t1}
	\label{tab:simple}\caption{Standard notation in manifold optimization.}
	\centering
	\begin{tabular}{llllllll}
		\toprule
		Notations & Meaning \\
		\midrule
		$D f$    & Tangent map of mapping $f$     \\	
		$D_x$    & Euclidean Derivative at $x$    \\
		$J$ & Jacobian matrix of a mapping \\
		$\sigma(A)$	& Singular value (non-negative real number) of matrix $A$\\	
		$\operatorname{det}(M) $ & determinant of mapping $M$ with an arbitrary basis\\
		$\operatorname{rank}(f) $ & rank of mapping $f$ with an arbitrary basis\\
		$P_{\mathcal{M}} (x) $    & Projecting $x$ to manifold $\mathcal{M}$     \\
		$T_x \mathcal{M}$ & Tangent space at $x$ \\
		$R_x[u]$ & Retraction at point $x$ with tangent vector $u$\\
		$gradf(x)$ & Riemannian gradient of $f(x)$ at $x$\\
		$D gradf(x)$ & Direction derivative of Riemannian gradient $gradf(x)$ at $x$\\
		\bottomrule
	\end{tabular}
\end{table}

\section{Correction}

\subsection{Incorrectness}

Notice that
\[
\mathrm{D} g(x)=\mathrm{D} P_{\mathcal{M}}\left(x-\alpha P_{T_x \mathcal{M}} \nabla f(x)\right)\left(I-\alpha \mathrm{D}\left(P_{T_x \mathcal{M}} \nabla f\right)(x)\right).
\]
In order to prove Proposition \ref{thm2:2},
Lee et al. \cite{lee19} introduced
\[
h_x (\alpha)=\operatorname{det}\left(\mathrm{D} g(x)\right),
\]
which is the determinant of the matrix representing $Dg(x)$ with respect to a basis.
Let $(U, \phi)$ and $(V, \psi)$ be charts for $x$ and $g(x)$ on manifold $\mathcal{M}$, respectively, that is,
\[
\phi : U\subseteq \mathcal{M} \rightarrow R^{{\rm dim}(\mathcal{M})}, x\in U,~ \psi: V\subseteq \mathcal{M} \rightarrow R^{{\rm dim}(\mathcal{M})}, g(x)\in V.
\]
Then we can calculate $\operatorname{det} (D g)$ under differential structures $(U, \phi)$ and $(V, \psi)$:
\begin{equation}\label{detg}
\operatorname{det} (D g |_{x}) = \operatorname{det}( J(\psi \circ g \circ \phi^{-1}) |_{\phi(x)}),
\end{equation}
where $J|_{\phi(x)}$ denotes  Jacobian matrix at $\phi(x) \in R^{{\rm dim}(\mathcal{M})}$.

Lee et al. \cite{lee19} first proved that
\[
h_x (0) = \operatorname{det}\left( \mathrm{D} P_{\mathcal{M}}\left(x\right) \right)=
\operatorname{det}\left( P_{T_x \mathcal{M}} \right) = 1,
\]
where the second equality holds since $\mathrm{D} P_{\mathcal{M}}\left(x\right)|_{x\in \mathcal{M}}=P_{T_x \mathcal{M}}$ due to Lewis and Malick \cite[lemma 4]{lewis08}.

For any compact smooth manifold $\mathcal{M}$, there is a $r>0$ such that $P_{\mathcal{M}}$ is unique and smooth in a neighborhood of radius $r$ \cite{absil12}. Thus, for any $\alpha<r/\max_{x\in \mathcal{M}}\|\nabla f(x)\|$, it holds that $P_{\mathcal{M}}\left(x-\alpha P_{T_x \mathcal{M}} \nabla f(x)\right)$ is differentiable.
As shown by Lee et al. \cite[Prop.9]{lee19}  that
\begin{equation}\label{hneq0}
h_x (\alpha)> 0,~\forall   \alpha<C_{\mathcal{M},f}\triangleq\min
\left(\frac{r}{\max_{x\in \mathcal{M}}\|\nabla f(x)\|},~\frac{1}{B}
\right),
\end{equation}
where
\begin{equation}\label{hneq1}
		B:=\max _{x \in \mathcal{M}, \alpha<\frac{r}{\max _{x \in \mathcal{M}}\|\nabla f(x)\|}}\left|
D_{\alpha}h_{x}(\alpha)\right|<\infty.
\end{equation}
It implies from \eqref{hneq0} that $\mathrm{D} g$ is invertible, which completes the proof of Proposition \ref{thm2:2}.

Lee et al.'s above proof implicitly assumes that $h_{x}(\alpha)$ is differentiable for any $\alpha<r/\max_{x\in \mathcal{M}}\|\nabla f(x)\|$.
Unfortunately, this is in general incorrect.
As shown in the following counterexample,
$h_x (\alpha)$ could be even noncontinuous in the feasible interval.

\begin{example}\label{ex1}
Let $\mathcal{M}$ be unit sphere $S^{n-1}$. The projection operator is $
	P_{\mathcal{M}}\left(x\right)= x/||x||$ and the projection onto the tangent space of $\mathcal{M}$ is $P_{T_x M}= I-xx^T$.
	
	Choose local parameterizations $(U,\phi)$ and $(V,\psi)$ such that:
	\begin{eqnarray*}
		& & U = S^{n-1}\setminus \{(0,0,\cdots,0,1)\},~ V=S^{n-1}\setminus\{(0,0,\cdots,0,-1)\}, \\
		& & \phi(x_1,x_2,\cdots,x_n) = \left(\dfrac{x_1}{1-x_{n}},\dfrac{x_2}{1-x_{n}},\cdots,\dfrac{x_{n-1}}{1-x_{n}}\right),\\
		& & \psi(x_1,x_2,\cdots,x_n) = \left(\dfrac{x_1}{1+x_{n}},\dfrac{x_2}{1+x_{n}},\cdots,\dfrac{x_{n-1}}{1+x_{n}}\right).
	\end{eqnarray*}
It is not difficult to see that
	\begin{equation}
	D\phi \ne D\psi. \label{phipsi}
	\end{equation}
Let $x,\alpha_0$ be such that $x\in U$ and $g(x)|_{\alpha=\alpha_0} = (0,0,\cdots,0,1)$. Then, for $\alpha<\alpha_0$, we have
	\[
		h_x(\alpha) = \operatorname{det} (D\phi Dg(x) D\phi^{-1}).
	\]
However, when $\alpha=\alpha_0$, we can only use $(V,\psi)$ as local parameterization, i.e.,
	\[
		h_x(\alpha_0) = \operatorname{det} (D\psi Dg(x) D\phi^{-1}).
	\]
The discontinuity of $h_x(\alpha)$ follows from \eqref{phipsi}.
On the other hand, in this case, the radius $r$ in \eqref{hneq1} can be set arbitrarily large. It follows that  the constants $B$ \eqref{hneq1} is not well-defined. Consequently, the claim \eqref{hneq0} is incorrect.
\end{example}

\begin{remark}
It should be noted that if we always take $(V,\psi)$ as the local parameterization in Example \ref{ex1}, then $h_x(\alpha)$ becomes differentiable. It remains unknown whether there are proper local parameterizations to guarantee the differentiability  of  $h_x(\alpha)$. Even if the answer is positive, it is difficult to determine step-size and find the correct local parameterizations.
\end{remark}

%

\subsection{A corrective proof}

The local diffeomorphism (embedding) between manifolds is an {\it immersion}, which is sufficient to check the non-singularity of $Dg(x)$.
The basic idea of our correction in showing  that $g(x)$ is immersion is  to verify $\operatorname{rank} (Dg(x)) = \operatorname{dim} \mathcal{M}$ (or $\sigma_{\min}(Dg(x)|_{T_x \mathcal{M}})>0$) rather than $\det(Dg(x))\neq0$, though they are equivalent. 

Notice that
\[
J(\psi \circ g \circ \phi^{-1}) |_{\phi(x)} = J(\psi)|_{g(x)}J(g(x))J(\phi^{-1})|_{\phi(x)},
\]
where$J(\phi^{-1}): \mathbb{R}^{{\rm dim}(\mathcal{M})} \rightarrow T_x \mathcal{M}$. Since $\phi$ and $\psi$ are local diffeomorphism, $J(\phi^{-1})$ and $J(\psi)$ are both of full rank. 

We first present a useful technical lemma, which shows that $x$ and the tangent vector $u$ can be decoupled when $D g$ is restricted in the tangent space.

\begin{lemma}\label{thm0}
For $\alpha>0$, we have
	\begin{eqnarray*}
		J g(x) [\xi] = D_u R_x[u] (I - \alpha D grad f(x)) [\xi], ~ \xi \in T_x \mathcal{M}.
	\end{eqnarray*}
where $u=-\alpha grad f(x)$ and $R_x[u] = P_{\mathcal{M}}(x+u)$ is the projection-like retraction.
\end{lemma}

\begin{proof}\smartqed
 For any $\xi \in T_x \mathcal{M}$, we have
\begin{eqnarray}
D_x g(x) [\xi] &=& D_x P_{\mathcal{M}}(x-\alpha grad f(x)) [\xi]\nonumber\\
&&\lim_{t \rightarrow 0} \dfrac{P_{\mathcal{M}}(x+t\xi-\alpha grad f(x+t\xi))-P_{\mathcal{M}}(x-\alpha grad f(x))}{t}\nonumber\\
&=& \lim_{t \rightarrow 0} \dfrac{P_{\mathcal{M}}(x+t\xi-\alpha grad f(x+t \xi)) - P_{\mathcal{M}}(x+t\xi-\alpha grad f(x))}{t}\nonumber\\ &&~~~~+\dfrac{P_{\mathcal{M}}(x+t\xi-\alpha grad f(x)) - P_{\mathcal{M}}(x-\alpha grad f(x))}{t}\nonumber\\
&=& \lim_{t \rightarrow 0} \dfrac{D P_{\mathcal{M}}(x+t\xi-\alpha grad f(x))(\alpha grad f(x)-\alpha grad f(x+t\xi)) }{t}\nonumber\\
&& ~~~~+D_u R_x [u] \xi |_{u=-\alpha grad f(x)}\nonumber\\
&=& D_y P_{\mathcal{M}}(y)(-\alpha D gradf(x))[\xi]+ D_u R_x [u] \xi|_{u=-\alpha grad f(x)}, \label{plem1}
\end{eqnarray}
where $y$ denotes $x-\alpha grad f(x)$.

Since $D_u P_{\mathcal{M}}(x+u)=D_{u+x} P_{\mathcal{M}}(x+u)$ when $x$ is fixed and $u$ is independent of $x$, we have
\begin{equation}\label{plem2}
D_y P_{\mathcal{M}}(y) = D_u P_{\mathcal{M}}(x+u) =  D_u R_x [u]|_{u=-\alpha grad f(x)}
\end{equation}
Combining \eqref{plem1} with \eqref{plem2} yields that
\[
J g(x) [\xi]= D_u R_x [u](I-\alpha D gradf(x))[\xi]|_{u=-\alpha grad f(x)},
\]
which completes the proof.
\qed
\end{proof}

Let $\mathcal{M}$ be a compact manifold and $r$ be a positive constant so that  $P_{\mathcal{M}}$ is unique and smooth in a neighborhood of radius $r$ \cite{absil12}. Let $R_x[u]$ be a projection-like retraction at any $x\in\mathcal{M}$.
Absil \cite{absil09} has shown that $DR_x[u]$ is smooth with respect to $u\in T_x \mathcal{M}$. We now further introduce the Retraction L-smooth property for  $DR_x[u]$.

\begin{assumption}
[Retraction L-smooth]\label{ass}  $R_x[u]$ is assumed to be Retraction L-smooth around $0$ if there is a constant $L>0$ such that
	\[
	\|DR_x[u]-DR_x[0]\|\leq L\|u\|, ~\forall x\in \mathcal{M}, ~\forall u \in T_x \mathcal{M},~ \|u\|<r.
	\]
\end{assumption}

\begin{remark}
Clearly, the Retraction L-smooth property naturally holds for any smooth projection-like retraction. No access to information beyond retraction (such as $f(x)$) is necessary to determine the Lipschitz constant $L$.
\end{remark}

Under the above Retraction L-smooth assumption,
base on Lemma \ref{thm0} and the following two well-known lemmas, we can establish the main result.

\begin{lemma}
[Poincar\'e separation theorem]\label{thm4} Let  $A\in R^{n\times n}$ be symmetric and $U\in R^{n\times k}$ be column-wise orthogonal, i.e., $U^TU=I_k$. It holds that
		\[
			\lambda_{n-k+i}(A) \leqslant \lambda_{i}\left(U^{*} A U\right) \leqslant \lambda_{i}(A), \quad i=1, \cdots, k.
		\]
	\end{lemma}

\begin{lemma}[Weyl's eigenvalue perturbation inequality]\label{thm5}
		  For any matrices $M,\Delta\in R^{n\times p}$, it holds that
		\[
			\left|\sigma_{k}(M+\Delta)-\sigma_{k}(M)\right| \leq \sigma_{1}(\Delta).
		\]
	\end{lemma}

\begin{theorem}\label{thm3}
Let $\mathcal{M}$ be a compact submanifold in $R^n$. Under the Retraction L-smooth assumption, $g(x)$ is an immersion (or equivalently, $\operatorname{det}(\mathrm{D} g(x))\neq 0$) as long as the step-size $\alpha$ satisfies that
\begin{equation}\label{albd}
	\alpha< \min\left(\dfrac{1}{2H},\frac{1}{3LG}, \frac{r}{G} \right),
	\end{equation}
where $H=\max_{x\in\mathcal{M}}\|Dgrad f(x)\|$ and $G=\max_{x\in\mathcal{M}}\|gradf(x)\|$.
\end{theorem}

\begin{proof}
According to Lemma \ref{thm0}, for any $\xi \in T_x \mathcal{M}$, we have
	\[
		D g(x) [\xi]= D_u R_x [u](I-\alpha D gradf(x))[\xi]|_{u=-\alpha grad f(x)}.
	\]
Let $m=\operatorname{dim}(\mathcal{M})$. Then $m\le n$. Let $\{e_1,\cdots,e_m\}$ be  an orthonormal basis of the tangent space $T_x \mathcal{M}$.
Denote $E$ by the matrix $(e_1,e_2,\cdots,e_m)\in R^{n\times m}$.
 Under this parameterazation,  $Dg(x) [\xi] $ is of full rank if and only if
	\begin{equation}\label{pf}
\operatorname{\sigma_{min} }\left( D_u R_x [u](I-\alpha D gradf(x))E|_{u=-\alpha grad f(x)}\right) > 0.
	\end{equation}
	
First, we have	
\begin{eqnarray}
 &&\operatorname{\sigma_{min}^2}\left( D_u R_x [0](I-\alpha D gradf(x))E \right)
	  =  \operatorname{\sigma_{min}^2}\left( P_{T_x \mathcal{M}} (I-\alpha D gradf(x))E\right) \nonumber \\
	 &=& \operatorname{\sigma_{min}^2}\left( (I-\alpha P_{T_x \mathcal{M}} D gradf(x))E\right)
	  = \operatorname{\sigma_{min}^2}\left( (I-\alpha Hess f(x))E\right)\nonumber\\
&=&\operatorname{\lambda}_{m} \left(E^T (I-\alpha Hess f(x))^T(I-\alpha Hess f(x))E\right)\nonumber\\
&\ge&\operatorname{\lambda}_{n}   ((I-\alpha Hess f(x))^T(I-\alpha Hess f(x)))  \nonumber\\
&=& \min_{i} (1-\alpha \lambda_{i}(Hess f(x)))^2
 \ge  (1-\alpha\|Hess f(x)\|)^2 \ge 1/4, \label{rzero}
	\end{eqnarray}
where the first inequality is due to Lemma \ref{thm4}, and the last inequality
holds since it follows from \eqref{albd} that
$1/2- \alpha\|Hess f(x)\|>0$.

%

Take $u=-\alpha gradf(x)$.
For $\alpha$ satisfying \eqref{albd}, we have
\begin{equation}\label{uG}
\|u\|=\|-\alpha gradf(x)\|\le \alpha G <r.
\end{equation}
	Let $V_D:=(I-\alpha D gradf(x))E$. Then we have
	\begin{eqnarray}
	&&\sigma_{\min}(D_u R_x[0]V_D) - \sigma_{\min}(D_u R_x[u]V_D) \le
	||D_u R_x[u]V_D-D_u R_x[0]V_D||  \nonumber\\
&&\le  ||D_u R_x[u]-D_u R_x[0]||\cdot ||V_D||
	 \le  L ||u|| \cdot  ||V_D||\le L G \alpha ||V_D||,\label{url}
	\end{eqnarray}
where the first inequality holds according to Lemma \ref{thm5}, the third inequality follows from the Retraction L-smooth assumption, and the last inequality is due to	\eqref{uG}. It follows from \eqref{url} that
	\begin{equation}\label{dlg:1}
		\sigma_{\min}(D_u R_x[u]V_D)\ge\sigma_{\min}(D_u R_x[0]V_D)-L G \alpha \|V_D\|.
	\end{equation}
According to Lemma \ref{thm4}, we have
	\begin{eqnarray}
		\|V_D\|&=& \sigma_{\max}((I-\alpha D gradf(x))E)
 \le  \sigma_{\max}(I-\alpha D gradf(x)) \nonumber \\
&\le& 1+\alpha \sigma_{\max}( D gradf(x))=1+\alpha\|D gradf(x)\|\le 3/2,\label{dlg:2}
	\end{eqnarray}
where the last inequality follows from \eqref{albd}.
%
Substituting \eqref{rzero} and \eqref{dlg:2} into \eqref{dlg:1} yields that
	\[
	\sigma_{\min}(D_u R_x[u]V_D)\ge\frac{1}{2} - \frac{3}{2}L G \alpha >0,
	\]
where the last inequality holds due to \eqref{albd}.
Thus, we have proved \eqref{pf}. Then we have
$\operatorname{det}(\mathrm{D} g(x))\neq 0$ and hence $g(x)$ is an immersion.
\qed
\end{proof}


%
%

\end{document}